\documentclass[a4paper,11pt]{amsart}
\usepackage{amsmath,amsthm,amsfonts,amssymb,graphicx,mathtools,enumerate}
\usepackage[usenames,dvipsnames,svgnames,table]{xcolor}
\usepackage{thmtools}
\usepackage{thm-restate}
\usepackage{xpatch}
\usepackage{apptools}
\usepackage[foot]{amsaddr}
\usepackage[margin=35mm]{geometry}
\usepackage[pdftitle={Defective and Clustered  Choosability of Sparse Graphs},colorlinks=true,citecolor=black,linkcolor=black,urlcolor=blue!80!black]{hyperref}
\usepackage[noabbrev,capitalise]{cleveref}
\usepackage[numbers,sort&compress]{natbib}
\allowdisplaybreaks
\makeatletter
\xpatchcmd{\thmt@restatable}% Edit \thmt@restatable
{\csname #2\@xa\endcsname\ifx\@nx#1\@nx\else[{#1}]\fi}% Replace this code
{\ifthmt@thisistheone\csname #2\@xa\endcsname\ifx\@nx#1\@nx\else[{#1}]\fi\else\csname #2\@xa\endcsname\fi}% with this code
{}{} % execute code for success/failure instances
\makeatother
\makeatletter
\def\NAT@spacechar{~}
\makeatother
\theoremstyle{plain}
\newtheorem{thm}{Theorem}
\newtheorem{lem}[thm]{Lemma}
\newtheorem{cor}[thm]{Corollary}
\newtheorem{claim}{Claim}

\crefname{lem}{Lemma}{Lemmas}
\crefname{thm}{Theorem}{Theorems}
\crefname{cor}{Corollary}{Corollaries}
\crefname{prop}{Proposition}{Propositions}
\crefname{conj}{Conjecture}{Conjectures}
\crefname{claim}{Claim}{Claims}
\crefname{openproblem}{Open Problem}{Open Problems}
\crefformat{equation}{(#2#1#3)}
\Crefformat{equation}{Equation #2(#1)#3}
\newcommand{\arXiv}[1]{arXiv:\,\href{http://arxiv.org/abs/#1}{#1}}
\newcommand{\msn}[1]{MR:\,\href{http://www.ams.org/mathscinet-getitem?mr=MR#1}{#1}}
\newcommand{\doi}[1]{doi:\,\href{http://dx.doi.org/#1}{#1}}
\DeclarePairedDelimiter\ceil\lceil\rceil
\DeclarePairedDelimiter\floor\lfloor\rfloor
\DeclareMathOperator{\mad}{mad}

\renewcommand{\geq}{\geqslant}
\renewcommand{\leq}{\leqslant}

\newcommand{\GG}{\mathcal{G}}
\begin{document}
\title[Defective and Clustered Choosability of Sparse Graphs]{Defective and Clustered  Choosability\\ of Sparse Graphs}
\author{Kevin Hendrey}
\address[K. Hendrey]{School of Mathematical Sciences, Monash University, Melbourne, Australia}
\email{kevinhendrey@gmail.com}
\author{David R. Wood}
\address[D. R. Wood]{School of Mathematical Sciences, Monash University, Melbourne, Australia}
\email{david.wood@monash.edu}
\date{\today}

\begin{abstract} 
An (improper) graph colouring has \emph{defect} $d$ if each monochromatic subgraph has maximum degree at most $d$, and has \emph{clustering} $c$ if each monochromatic component has at most $c$ vertices.  This paper studies defective and clustered list-colourings for graphs with given maximum average degree.  We prove that every graph with maximum average degree less than $\frac{2d+2}{d+2} k$ is $k$-choosable with defect $d$. This improves upon a similar result by Havet and Sereni [\emph{J. Graph Theory}, 2006]. For clustered choosability of graphs with maximum average degree $m$, no $(1-\epsilon)m$ bound on the number of colours was previously known. The above result with $d=1$ solves this problem. It implies that every graph with maximum average degree $m$ is $\floor{\frac{3}{4}m+1}$-choosable with clustering 2. This extends a result of Kopreski and Yu [\emph{Discrete Math.}, 2017] to the setting of choosability. We then prove two results about clustered choosability that explore the trade-off between the number of colours and the clustering. In particular, we prove that every graph with maximum average degree $m$ is $\floor{\frac{7}{10}m+1}$-choosable with clustering $9$, and is $\floor{\frac{2}{3}m+1}$-choosable with clustering $O(m)$. As an example, the later result implies that every biplanar graph is 8-choosable with bounded clustering. This is the best known result for the clustered version of the earth-moon problem. The results extend to the setting where we only consider the maximum average degree of subgraphs with at least some number of vertices. Several applications are presented. 
\end{abstract}

\maketitle

%%%%%%%%%%%%%%%%%%%%%%
\section{Introduction}
\label{Intro}

This paper studies improper colourings of sparse graphs, where sparsity is measured by the following standard definition. The \emph{maximum average degree} of a graph $G$, denoted by $\mad(G)$, is the maximum, taken over all subgraphs $H$ of $G$, of the average degree of $H$. We consider improper colourings with bounded monochromatic degree or with bounded monochromatic components, for graph classes with bounded maximum average degree. 
We now formalise these ideas. A \emph{colouring} of a graph $G$ is a function that assigns a colour to each vertex.  In a coloured graph $G$, the \emph{monochromatic subgraph} of $G$ is the spanning subgraph consisting of those edges whose endpoints have the same colour. A colouring has \emph{defect} $k$ if the monochromatic subgraph has maximum degree at most $k$; that is, each vertex $v$ is adjacent to at most $k$ vertices of the same colour as $v$. A connected component of the monochromatic subgraph is called a \emph{monochromatic component}. A  colouring has \emph{clustering} $k$ if each monochromatic component has at most $k$ vertices. Of course, a colouring is proper if and only if it has defect $0$ or clustering $1$. 

Our focus is on minimising the number of colours, with small defect or small clustering as a secondary goal. This viewpoint leads to the following definitions. The \emph{defective chromatic number} of a graph class $\GG$ is the minimum integer $k$ such that  for some integer $d$, every graph in $\GG$ is $k$-colourable with defect $d$. The \emph{clustered chromatic number} of a graph class $\GG$ is the minimum integer $k$ such that  for some integer $c$, every graph in $\GG$ is $k$-colourable with clustering $c$. 

The above definitions extend in the obvious way to list-colourings and choosability. %, as formalised in \cref{Definitions}. 
A \emph{list-assignment} for a graph $G$ is a function $L$ that assigns a set $L(v)$ of colours to each vertex $v\in V(G)$. 
A list-assignment $L$ is a \emph{$k$-list-assignment} if $|L(v)|\geq k$ for each vertex $v\in V(G)$. 
An \emph{$L$-colouring} is a colouring of $G$ such that  each vertex $v\in V(G)$ is assigned a colour in $L(v)$.  
%The \emph{choice number} of a graph $G$ is the minimum integer $k$ such that $G$ is $L$-colourable for every $k$-list-assignment $L$ of $G$.
%For a list-assignment $L$ of a graph $G$ and integer $d\geq 0$, define $G$ to be \emph{$L$-colourable with defect $d$} if there is a  colouring of $G$ with defect $d$ such that  each vertex $v\in V(G)$ is assigned a colour in $L(v)$.  
Define $G$ to be \emph{$k$-choosable with defect $d$} if $G$ has an $L$-colouring with defect $d$ for every $k$-list-assignment $L$ of $G$. Similarly, 
%for an integer $c\geq 1$, $G$ is \emph{$L$-colourable with clustering $c$} if there is a  colouring of $G$ with clustering $c$ such that each vertex $v\in V(G)$ is assigned a colour in $L(v)$.  Define 
$G$ is \emph{$k$-choosable with clustering $c$} if $G$ has an $L$-colouring with clustering $c$ for every $k$-list-assignment $L$ of $G$. 

Defective and clustered (list-)colouring has been widely studied on a variety of graph classes, including: 
bounded maximum degree \citep{HST03,ADOV03}, 
planar \citep{CCW86,EH99,CK10}, 
bounded genus \citep{CCW86,Archdeacon87,CGJ97,Woodall11,CE16,EO16}, 
excluding a minor \citep{DN17,EKKOS15,vdHW18,OOW,LO17,NSSW}, 
excluding a topological minor \citep{DN17,OOW}, 
and excluding an immersion \citep{vdHW18}. See \citep{WoodSurvey} for a survey on defective and clustered colouring. 
All of these classes have bounded maximum average degree. 
Thus our results are more widely applicable than nearly all of the previous results in the field. That said, it should be noted that some of the existing results for more specific graph classes give better bounds on the number of colours or on the defect or clustering. Generally speaking, our results give the best known bounds for graph classes that have bounded maximum average degree, unbounded maximum degree, and have no strongly sub-linear separator theorem. Examples include graphs with given thickness, stack-number or queue-number. 

\subsection{Defective Choosability}

Defective choosability with respect to maximum average degree was previously studied by \citet{HS06}, who proved the following theorem. 
 
\begin{thm}[\citep{HS06}]
\label{HavetSereni}
For $d\geq 0$ and $k\geq 2$, every graph $G$ with $\mad(G)<k+\frac{kd}{k+d}$ is $k$-choosable with defect $d$. 
\end{thm}

%\begin{thm}[\citep{DKMR14}]
%Let $a,b,d$ be integers with $a + b > 0$ and $d > 0$. Every graph $G$ with 
%$$\mad(G) < a + b + \frac{da(a + 1)}{(a + d + 1)(a + 1) + ab}$$
%is $(a+b)$-colourable, such that $a$ colour classes have monochromatic degree $d$, and $b$ colour classes have monochromatic degree $0$. 
%\end{thm}
%This result with $b=0$ says:
%\begin{cor}[\citep{DKMR14}]
%For integers $k,d \geq 1$, every graph $G$ with 
%$$\mad(G) < k + \frac{dk(k + 1)}{(k + d + 1)(k + 1)}$$
%is $k$-colourable with defect $d$.
%\end{cor}
%With $d=1$:
%\begin{cor}[\citep{DKMR14}]
%For integers $k \geq 1$, every graph $G$ with 
%$$\mad(G) < k + \frac{k(k + 1)}{(k + 2)(k + 1)}$$
%is $k$-colourable with defect $d$.
%\end{cor}
%

Our first result improves on \cref{HavetSereni} as follows:

\begin{thm}[\S\ref{DefectiveChoosability}]
\label{MADdefect}
For $d\geq 0$ and $k\geq 1$, every graph $G$ with $\mad(G)<\frac{2d+2}{d+2}\, k$ is $k$-choosable with defect $d$.
\end{thm}

Note that the two theorems are equivalent for $k=2$. But for $k\geq 3$,  the assumption in \cref{MADdefect} is weaker than the corresponding assumption in \cref{HavetSereni}, thus \cref{MADdefect} is stronger than \cref{HavetSereni}.

\cref{HavetSereni} can be restated as follows: every graph $G$ with $\mad(G)=m$ is $k$-choosable with defect $\floor{\frac{ k( m - k)}{2k-m}}+1$, whereas  \cref{MADdefect} says that $G$ is $k$-choosable with defect $\floor{ \frac{m}{2k-m}} $. Both results require that $2k>m$, and the minimum value of $k$ for which either theorem is applicable is $k=\floor{\frac{m}{2} }+1$. In this case, \cref{MADdefect} gives a defect bound of $\floor{\frac{m}{2k-m}}$, which is an order of magnitude less than the defect bound of $(1+o(1))\frac{k^2}{2k-m}$ in  \cref{HavetSereni}. 
Note that \citet{HS06} gave a construction to show that no lower value of $k$ is possible. That is, for $m\in\mathbb{R}^+$, the defective chromatic number of the class of graphs with maximum average degree $m$ equals $\floor{\frac{m}{2} }+1$; also see \citep{WoodSurvey}. 

%\comment{Is the defect bound in \cref{MADdefect} best possible? Suppose that every graph $G$ with $\mad(G)<f(d)\, k$ is $k$-colourable with defect $d$. Is it true that $f(d) \geq \frac{2d+2}{d+2}$? I think the standard example almost proves this --- work out the details.}

See \citep{KKZ16,BK13,BKY13,KKZ14,BIMR12,BIMR11,BIMOR10,BI11,BI09a,BorKos11} for results about defective 2-colourings of graphs with given maximum average degree, where each of the two colour classes has a prescribed degree bound. Also note that \citet{DKMR14} proved a result analogous to \cref{HavetSereni,MADdefect} (with weaker bounds) for defective colouring of graphs with given maximum average degree, where in addition, a given number of colour classes are stable sets. 

%%%%%%%%%%%
\subsection{Clustered Choosability}

The following theorem, due to \citet{KY17}, is the only known non-trivial result for clustered colourings of graphs with given maximum average degree\footnote{\citet{KY17} actually proved the following stronger result: For $a\geq 1$ and $b\geq 0$, every graph $G$ with $\mad(G) < \frac43 a+b$ is $(a+b)$-colourable, such that $a$ colour classes have defect $1$, and $b$ colour classes are stable sets.}.

\begin{thm}[\citep{KY17}] 
\label{KY}
Every graph $G$ is $\floor{\frac{3}{4}\mad(G)+1}$-colourable with defect $1$, and thus with clustering $2$.
\end{thm}

There are no existing non-trivial results for clustered choosability of graphs with given maximum average degree. The closest such result, due to \citet{DN17}, says that for constants $\alpha,\gamma,\epsilon>0$, if a graph $G$ has at most $(k+1-\gamma)|V(G)|$ edges, and every $n$-vertex subgraph of $G$ has a balanced separator of order at most $\alpha n^{1-\epsilon}$, then $G$ is $k$-choosable with clustering some function of $\alpha$, $\gamma$ and $\epsilon$. Note that the number of colours here is roughly half the average degree of $G$. This result determines the clustered chromatic number of several graph classes, but for various other classes (that contain expanders) this result is not applicable because of the requirement that every subgraph has a balanced separator. 

%Prior to this work it was open whether the class of graphs with maximum average degree at most some number $m$ had clustered chromatic number at most $(1-\epsilon)m$ for some fixed $\epsilon$. 

%Determining the clustered chromatic number of graphs with maximum average degree $m$ was identified as an important open problem in \citep[Open Problem 30]{WoodSurvey}. 

\cref{MADdefect} with $d=1$ implies the above result of \citet{KY17} and extends it to the setting of choosability:
 
%Previously no $(1-\epsilon)m$ bound was known. The second contribution of this paper is to prove the first such bound. Indeed, this is an immediate corollary of \cref{MADdefect} with $d=1$, since every colouring with defect 1 has clustering 2. 

%\cref{HavetSereni} says nothing in the $d=1$ case. However,  implies:

\begin{thm}
\label{MADclusteringA}
Every graph $G$ is $\floor{\frac{3}{4}\mad(G)+1}$-choosable with defect $1$, and thus with clustering $2$.
\end{thm}

%Determining the clustered chromatic number of graphs with maximum average degree $m$ was identified as an important open problem in \citep[Open Problem 30]{WoodSurvey}. Previously no $(1-\epsilon)\mad(G)$ bound was known. So \cref{MADclusteringA} is the first non-trivial result for clustered colouring of graphs with given maximum average degree.

As an example of \cref{MADclusteringA}, it follows from Euler's formula that  toroidal graphs have maximum average degree at most 6, implying every toroidal graph is 5-choosable with defect $1$ and clustering $2$, which was first proved by \citet{DO18}. Previously, \citet{CGJ97}  proved that every toroidal graph is 5-colourable with defect $1$.

The following two theorems are our main results for clustered choosability. The first still has an absolute bound on the clustering, while the second has fewer colours but  at the expense of allowing the clustering to depend on the maximum average degree.

%The next theorem reduces the number of colours in \cref{MADclusteringA} at the expense of larger clustering.

\begin{restatable}[\S\ref{AbsoluteClustering}]{thm}{MADclusteringB}
\label{MADclusteringB}
Every graph $G$ is $\floor{\frac{7}{10}\mad(G)+1}$-choosable with clustering $9$.
\end{restatable}

%The next theorem further reduces the number of colours at the expense of allowing the clustering to depend on the maximum average degree.

\begin{thm}[\S\ref{ClusteredChoosabilityMAD}]
\label{MADclusteringD}
Every graph $G$ is $\floor{\frac23 \mad(G)+1}$-choosable with clustering $57\floor{\frac{2}{3}\mad(G)}+6$. 
\end{thm}

\cref{MADclusteringD} says that the clustered chromatic number of the class of graphs with maximum average degree $m$ is at most  $\floor{\frac{2m}{3}}+1$. This is the best known upper bound. The best known lower bound is $\floor{\frac{m}{2}}+1$; see \citep{WoodSurvey}. Closing this gap is an intriguing open problem. 

%%%%%%%%%%%%%%%%%%%%%%%
\subsection{Generalisation}

The above results generalise via the following definition. For a graph $G$ and integer $n_0\geq 1$, let $\mad(G,n_0)$ be the maximum average degree of a subgraph of $G$ with at least $n_0$ vertices, unless %$E(G)=\emptyset$ or 
$|V(G)|<n_0$, in which case $\mad(G,n_0):=0$. The next two results generalise \cref{MADdefect,MADclusteringD} respectively  with $\mad(G)$ replaced by $\mad(G,n_0)$, where the number of colours stays the same, and the defect or clustering bound also depends on $n_0$. 

\begin{restatable}[\S\ref{DefectiveChoosability}]{thm}{MADdefectExtension}
\label{MADdefectExtension}
For integers $d\geq 0$, $n_0\geq 1$ and $k\geq 1$, every graph $G$ with $\mad(G,n_0)<\frac{2d+2}{d+2}\, k$ is $k$-choosable with defect 
$d':=\max \{\lceil \frac{n_0-1}{k}\rceil-1,\,d\}$.
\end{restatable}

\begin{restatable}[\S\ref{ClusteredChoosabilityMAD}]{thm}{MADclusteringExtension}
\label{MADclusteringExtension}
For integers $d\geq 0$, $n_0\geq 1$ and $k\geq 1$, every graph $G$ with $\mad(G,n_0)<\frac{3}{2} k$ is $k$-choosable with clustering 
$c:= \max\{\ceil{ \frac{n_0-1}{k}}, \,57k-51\}$.
\end{restatable}

Note that \cref{MADdefectExtension} with $n_0=1$ is equivalent to \cref{MADdefect}, 
and \cref{MADclusteringExtension} with $n_0=1$ and $k=\floor{\frac23 \mad(G)}+1$ is equivalent to \cref{MADclusteringD}.

Graphs on surfaces provide motivation for this extension\footnote{The \emph{Euler genus} of the orientable surface with $h$ handles is $2h$. 
The \emph{Euler genus} of the non-orientable surface with $k$ cross-caps is $k$. 
The \emph{Euler genus} of a graph $G$ is the minimum Euler genus of a surface in which $G$ embeds.}.  Graphs with Euler genus $g$ can have average degree as high as $\Theta(\sqrt{g})$, the complete graph being one example. But such graphs necessarily have bounded size. In particular, Euler's formula implies that every $n$-vertex $m$-edge graph with Euler genus $g$ satisfies  $m< 3(n+g)$. Thus, for $\epsilon>0$, if $n\geq \frac{6}{\epsilon}g$ then $G$ has average degree $\frac{2m}{n} < 6+\epsilon$. In particular, $\mad(G,6g)<7$. 

Using this observation,  \cref{MADdefectExtension,MADclusteringExtension} respectively imply that graphs with bounded Euler genus are $4$-choosable with bounded defect and are $5$-choosable with bounded clustering. Both these results are actually weaker than known results.  In particular, several authors \cite{Archdeacon87,CGJ97,Woodall11,CE16} have proved that  graphs with bounded Euler genus are 3-colourable or 3-choosable with bounded defect. And \citet{DN17}  proved that graphs with bounded Euler genus are 4-choosable with bounded clustering. The proof of \citet{DN17} uses the fact that graphs of bounded Euler genus have strongly sub-linear separators. The advantage of our approach is that it works for graph classes that do not have sub-linear separator theorems. Graphs with given $g$-thickness are such a class \citep{DSW16}. We explore this direction in \cref{EarthMoon}.

\subsection{Clustered Choosability and Maximum Degree}

\citet*{ADOV03} and \citet*{HST03} studied clustered colourings of graphs with given maximum degree. \citet{HST03} proved that 
every graph with maximum degree $\Delta$ is $\ceil{\frac13(\Delta+1)}$-colourable with bounded clustering. Moreover, for some $\Delta_0$ and $\epsilon>0$, every graph with maximum degree $\Delta\geq\Delta_0$  is $\floor{\left(\frac13-\epsilon\right)\Delta}$-colourable with bounded clustering. For both these results, the clustering bound is independent of $\Delta$. 

Clustered choosability of graphs with given maximum degree has not been studied in the literature (as far as we are aware). As a by-product of our work for graphs with given maximum average degree we prove the following results for clustered choosability of graphs with given maximum degree. 

\begin{restatable}[\S\ref{ClusteredChoosabilityMaximumDegree}]{thm}{MaxDegreeChoose}
\label{MaxDegreeChoose}
Every graph $G$ with maximum degree $\Delta\geq 3$ is $\ceil{\frac13(\Delta+2)}$-choosable with clustering $\ceil{\frac{19}{2}\Delta}-17$.
\end{restatable}

\begin{thm}[\S\ref{AbsoluteClustering}]
\label{MaxDegreeChooseAbsolute}
Every graph $G$ with maximum degree $\Delta$  is $\ceil{\frac{2}{5}(\Delta+1)}$-choosable with clustering $6$.
\end{thm}

$\Delta=5$ is the first case in which the above results for clustered choosability are weaker than the known results for clustered colouring. In particular, \citet{HST03} proved that every graph with maximum degree $5$ is $2$-colourable with bounded clustering, whereas \cref{MaxDegreeChoose,MaxDegreeChooseAbsolute} only prove 3-choosability. It is open whether every graph with maximum degree $5$ is $2$-choosable with bounded clustering.

Finally, we remark that all our choosability results hold in the stronger setting of correspondence colouring, introduced by  \citet{DP18}. 

%%%%%%%%%%%%%%%%%%%%%%
\section{Definitions}
\label{Definitions}

Let $G$ be a graph with vertex set $V(G)$ and edge set $E(G)$. Let $\Delta(G)$ be the maximum degree of the vertices in $G$. 
For a subset $A\subseteq V(G)$ and vertex $v\in V(G)$, let $N_A(v) := N_G(v) \cap A$ and $\deg_A(v) := |N_A(v)|$.  
We sometimes refer to $|V(G)|$ as $|G|$. 

In a coloured graph, the \emph{defect} of a vertex is its degree in the monochromatic subgraph. Note that a colouring with defect $k$ also has defect $k+1$, but a vertex of defect $k$ does not have defect $k+1$.

\section{Defective Choosability and Maximum Average Degree }
\label{DefectiveChoosability}

This section proves our result for defective choosability (\cref{MADdefect}). 
The following lemma is essentially a special case of an early result of \citet{Lovasz66}. 

\begin{lem}
\label{DefectColour}
If $L$ is a list-assignment for a graph $G$, such that 
$$\deg_G(v)+1\leq|L(v)|(d+1)$$ 
for each vertex $v$ of $G$, then $G$ is $L$-colourable with defect $d$.
\end{lem}

\begin{proof}
Colour each vertex $v$ in $G$ by a colour in $L(v)$ so that the number of monochromatic edges is minimised. 
Suppose that some vertex $v$ coloured $\alpha$ is adjacent to at least $d+1$ vertices also coloured $\alpha$. 
Since $\deg(v)<|L(v)|(d+1)$, some colour $\beta\in L(v)\setminus\{\alpha\}$ is assigned to at most $d$ neighbours of $v$. Recolouring $v$ by $\beta$ reduces the number of monochromatic edges. This contradiction shows that no vertex $v$ is adjacent to at least $d+1$ vertices of the same colour as $v$. Thus the colouring has defect $d$. 
\end{proof}

\begin{cor}
\label{DefectColourDegree}
Every graph $G$ with $\Delta(G)+1 \leq k(d+1)$ is $k$-choosable with defect $d$.
\end{cor}

The next lemma is a key idea of this paper. It provides a sufficient condition for a partial list-colouring to be extended to a list-colouring of the whole graph. 

\begin{lem}
\label{Extend}
Let $L$ be a $k$-list-assignment of a graph $G$. 
Let $A,B$ be a partition of $V(G)$, where  $G[A]$ is $L$-colourable  with defect $d'$. 
If $d\leq d'$ and for every vertex $v\in B$, 
$$(d+1)\deg_A(v) + \deg_B(v)+1 \leq (d+1)k,$$ 
then $G$ is $L$-colourable with defect $d'$. 
\end{lem}

\begin{proof}
Let $\phi$ be an $L$-colouring of $G[A]$ with defect $d'$. 
For each vertex $v\in B$, let $L'(v):= L(v)\setminus \{\phi(x):x\in N_A(v\}$. 
Thus $|L'(v)| \geq k - \deg_A(v) \geq (\deg_B(v)+1)/(d+1)$.  
\cref{DefectColour} implies that $G[B]$ is $L$-colourable with  defect $d$. 
By construction, there is no monochromatic edge between $A$ and $B$. 
Thus  $G$ is $L$-colourable with defect $d'$. 
\end{proof}

We now prove our first main result, which is equivalent to \cref{MADdefect} when $n_0=1$. 

\MADdefectExtension*

\begin{proof}
We proceed by induction on $|V(G)|$. Let $L$ be a $k$-list-assignment for $G$. For the base case, suppose that $|V(G)|\leq n_0-1$. For each vertex $v$ of $G$, choose a colour in $L(v)$ so that each colour is used at most $\lceil \frac{|V(G)|}{k}\rceil$ times. We obtain an $L$-colouring with defect $\lceil \frac{n_0-1}{k}\rceil-1$. Now assume that $|V(G)|\geq n_0$. 

%Let $v_1,\dots,v_p$ be a maximal sequence of vertices in $G$, such that if 
%$A_i:=\{v_1,\dots,v_{i-1}\}$ and  $B_i:=V(G)\setminus A_i$, 
%then $(d+1)\deg_{A_i}(v_i) + \deg_{B_i}(v_i) \geq (d+1)k$. 
%
%% % % % %???
%Given a sequence $v_1,\dots,v_p$ of vertices in $G$ and an integer $i\in \{1,\dots, p\}$, define $A_i:=\{v_1,\dots,v_i\}$ and $B_i:=V(G)\setminus A_i$. Let $v_1,\dots,v_p$ be a maximal sequence of distinct vertices in $G$, such that for all $i\in \{1,\dots, p\}$, we have $(d+1)\deg_{A_i}(v_i) + \deg_{B_i}(v_i) \geq (d+1)k$.
%% % % % %???

Let $v_1,\dots,v_p$ be a maximal sequence of distinct vertices in $G$, such that for all $i\in \{1,\dots, p\}$, we have $(d+1)\deg_{A_i}(v_i) + \deg_{B_i}(v_i) \geq (d+1)k$, where $A_i:=\{v_1,\dots,v_i\}$ and $B_i:=V(G)\setminus A_i$.

First suppose that $p<|V(G)|$. Let $A:=\{v_1,\dots,v_p\}$ and  $B:=V(G)\setminus A$. 
By induction,  $G[A]$ is $L$-colourable with defect $d'$. 
By the maximality of $v_1,\dots,v_p$, for every vertex $v\in B$, we have  $(d+1)\deg_A(v) + \deg_B(v)+1 \leq (d+1)k$. 
By \cref{Extend}, $G$ is $L$-colourable with defect $d'$, and we are done. 

Now assume that $p=|V(G)|$. Thus
%, each vertex $v_i$ satisfies $d\deg_{A_i}(v_i) + \deg_G(v_i) \geq (d+1)k$.  Hence 
\begin{align*}
(d+2)|E(G)| 
& = \sum_{i=1}^{|V(G)|} d\deg_{A_i}(v_i) + \deg_G(v_i) \\
& = \sum_{i=1}^{|V(G)|} (d+1) \deg_{A_i}(v_i) + \deg_{B_i}(v_i) \\
& \geq (d+1)k |V(G)|.
\end{align*}
Since $|V(G)|\geq n_0$, we have
$\mad(G,n_0) \geq \frac{2|E(G)|}{ |V(G)|} \geq \frac{2d+2}{d+2}\,k$, which is a contradiction. 
\end{proof}

%
%It follows from Euler's formula that every $n$-vertex graph with Euler genus $g$ and girth $h$ has less than $\frac{h}{h-2}(n+g)$ edges. 
%
%planar: $\mad(G) < 2 + \frac{4}{h-2}$. 
%
%$2 + \frac{4}{h-2} \leq (2 - \frac{2}{d+2}) k $
%\quad
%
%$k=2$
%$ \frac{2}{h-2} \leq 1 - \frac{2}{d+2} $
%
%$k=2$ and $d=1$
%$ 8 \leq h$
%
%
%
%Every graph $G$ with $\mad(G,n_0)<\frac{2d+2}{d+2}\, k$ is $k$-choosable with defect 
%$d':=\max \{\lceil \frac{n_0-1}{k}\rceil-1,d\}$.
%
%\begin{thm}
%\label{MADdefect+}
%For $d,n_0\geq 0$ and $k\geq 1$, every graph $G$ with $\mad(G,n_0)<\frac{2d+2}{d+2}\, k$ is $k$-choosable with defect 
%$d':=\max \{\lceil \frac{n_0-1}{k}\rceil-1,d\}$.
%\end{thm}

%%%%%%%%%%%%%%%%%%%%%%%%%
\section{Using Independent Transversals}
\label{IndependentTransversals}

This section introduces a useful tool, called ``independent transversals'', which have been previously used for clustered colouring by \citet{ADOV03} and \citet{HST03}. \citet{Haxell01} proved the following result.

\begin{lem}[\citep{Haxell01}]
\label{Haxell}
Let $G$ be a graph with maximum degree at most $\Delta$. 
Let $V_1,\dots, V_n$ be a partition of $V(G)$, with $|V_i|\geq 2\Delta$ for each $i\in[n]$. 
Then $G$ has a stable set $\{v_1,\dots,v_n\}$ with $v_i \in V_i$ for each $i\in[n]$.
\end{lem}

\begin{lem}
\label{ApplyIndTrans}
Let $\Delta\geq 3$ and let $G$ be a graph of maximum degree at most $\Delta$. If $H$ is a subgraph of $G$ with $\Delta(H)\leq 2$, then $G$ has a stable set $S\subseteq V(H)$ of vertices of degree 2 in $H$ with the following properties:
\begin{enumerate}
\item every subpath of $H$ with at least $3\Delta-6$ vertices that contains  a vertex with degree 1 in $H$ contains at least one vertex in $S$,
\item every subpath  of $H$ with at least $5\Delta-9$ vertices that contains  a vertex with degree 1 in $H$ contains at least two vertices in $S$,
\item every connected subgraph $C$ of $H$ with at least $\ceil{\frac{19}{2}\Delta}-16$ vertices contains at least three vertices in $S$.
\end{enumerate}
\end{lem}

\begin{proof}
Consider each cycle component $C$ of $H$ with $|C|\geq 8\Delta-12$. Say $|C|=(2\Delta-3)a+b$, where $a\geq 4$ and $b\in[0,2\Delta-4]$. Partition $C$ into subpaths $A_1B_1A_2B_2\dots A_aB_a$ where $|A_i|=2\Delta-4$ and $|B_i|\in [1, 1+\ceil{\frac{b}{a}}]$ for $i\in[a]$. Note that $|B_i|\leq 1+\ceil{\frac{b}{a}}\leq \ceil{\frac12 \Delta}$.

Now consider each path component $P$ of $H$ with $|P|\geq 2\Delta-4$. Say $|P|=(2\Delta-3)a+b-1$, where $a\geq 1$ and $b\in[0,2\Delta-4]$. Partition $P$ into subpaths $B_0A_1B_1\dots A_aB_a$ where $|A_i|=2\Delta-4$ for $i\in[a]$, $|B_i|=1$ for $i\in [a-1]$, and $|B_i|\leq \ceil{\frac{b}{2}}$. 

Let $\mathcal{A}$ be the set of all such paths $A_i$ taken over all the components of $H$. 
Let $G':=G[ \bigcup_{A\in\mathcal{A}} V(A)]-E(H)$.
Then $\mathcal{A}$ gives a partition of $V(G')$ into parts, each of which has exactly $2\Delta-4$ vertices, and $\Delta(G')\leq \Delta-2$. 
By \cref{Haxell}, $G'$ has a stable set $S$ that contains  exactly one vertex in each path in $\mathcal{A}$. 
By construction, every vertex in $S$ has degree 2 in $H$ and $S$ is a stable set in $H$, so $S$ is a stable set in $G$.

Let $P$ be a path in $H$ that contains  a vertex of degree 1 in $H$. Then $H$ is subpath of some component path $P'$ of $H$. If $P$ contains at least $3\Delta-6$ vertices, then $|P'|= (2\Delta-3)a+b-1$ where $a\geq 1$ and $b\in[0,2\Delta-4]$. Now, using our previous notation, $|B_0A_1|\leq \Delta-2+2\Delta-4=3\Delta-6\leq |P|$ and  $|A_aB_a|\leq \Delta-2+2\Delta-4=3\Delta-6\leq |P|$, so $P$ is not a proper subpath of $B_0A_1$ or of $B_aA_a$. Hence $P$ contains every vertex of $A_i$ for some $i\in \{1,a\}$, so $P$ contains a vertex in $S$. 

If $P$ contains at least $5\Delta-9$ vertices, then $|P'|= (2\Delta-3)a+b-1$ where $a\geq 2$ and $b\in[0,2\Delta-4]$. Now, $|B_0A_1B_1A_2|\leq \Delta-2+2(2\Delta-4)+1=5\Delta-9\leq |P|$ and  $|A_{a-1}B_{a-1}A_aB_a|\leq 5\Delta-9\leq |P|$, so $P$ is not a proper subpath of $B_0A_1B_1A_2$ or of $A_{a-1}B_{a-1}A_aB_a$. Hence $P$ contains every vertex $A_i$ and of $A_{i+1}$ for some $i\in \{1,a-1\}$, so $P$ contains two vertices in $S$.

%If $P$ contains at least $\ceil{\frac{19}{2}\Delta}-16$ vertices, then $|P|= (2\Delta-3)a+b-1$ where $a\geq 3$ and $b\in[0,2\Delta-4]$. Now, $|V(B_0A_1,\dots, A_3)|\leq \Delta-2+3(2\Delta-4)+2=7\Delta-12< |P|$ and  $|V(B_aA_a,\dots, A_{a-2})|\leq 7\Delta-12< |P|$, so $P$ is not a proper subpath of $B_0A_1,\dots,A_3$ or of $B_aA_a,\dots A_{a-2}$. Hence $P$ contains every vertex $A_i$, $A_{i+1}$ and $A_{i+2}$ for some $i\in \{1,a-2\}$, so $P$ contains three vertices in $S$.

%no more than $1$ vertices in $S$. Then, either $|V(P)|\leq |V(P')|\leq 2\Delta-4$, or there are at most $k$ paths $A_i\in \mathcal{A}$ such that $A_i$ is a subpath of $P$, and $|V(P)|\leq |B_a|+(2\Delta-3)k+2\Delta-5\leq (2k+3)\Delta-3k-7$ or $|V(P)|\leq |B_0|+(2\Delta-4)k+(k-1)+|B_a|\leq (2k+2)\Delta-3k-3$.

Suppose for contradiction there is a connected subgraph $C$ of $H$ on $\ceil{\frac{19}{2}\Delta}-16$ vertices with at most two vertices in $S$. By the definition of $S$, there are at most two paths $A_i\in \mathcal{A}$ with $V(A_i)\subseteq V(C)$. If $C$ is contained in some path component of $H$, then $C$ is a proper subpath of $A_jB_j A_{j+1}B_{j+1} A_{j+2}B_{j+2} A_{j+3}$ for some $j\in \{0,\dots,a-3\}$, where we take $A_0$ and $A_{a+1}$ to be the empty path for simplicity (so $|A_0B_0|=|B_0|\leq \Delta-2$ and $|B_aA_{a+1}|=|B_a|\leq \Delta-2$). Now $|A_jB_j A_{j+1}B_{j+1} A_{j+2}B_{j+2} A_{j+3}|\leq 4(2\Delta-4)+3\leq \ceil{\frac{19}{2}\Delta}-17$.

If $C$ is contained in some cycle component of $H$, we may assume without loss of generality that $C$ is a subpath of the path $A_1B_1A_2B_2A_3B_3 A_4$, and does not contain every vertex of $A_1$ and does not contain every vertex of $A_4$. Thus, $|V(C)|\leq |A_1B_1A_2B_2A_3B_3A_4|-2\leq 4(2\Delta-4)+3\ceil{\frac12 \Delta}-2\leq \ceil{\frac{19}{2}\Delta}-17$, a contradiction.
%Let $C$ be a connected subgraph of $H$ with at least $\ceil{\frac{19}{2}\Delta}-16$ vertices. Then $H$ is contained in some cycle component $C'$ of $H$ with at least $8\Delta-12$ vertices, or in some path component $P$ of $H$ on at least $4(2\Delta-3)$ vertices. By construction, there are at least three paths $A_i\in \mathcal{A}$ such that $V(A_i)\subseteq V(C)$, and so $C$ contains at least three vertices of $S$. 
\end{proof}

\section{Clustered Choosability and Maximum Degree}
\label{ClusteredChoosabilityMaximumDegree}

This section proves our first result about clustered choosability of graphs with given maximum degree (\cref{MaxDegreeChoose}). The preliminary lemmas will also be used in subsequent sections. 

\begin{lem}
\label{NewDefectColour}
If $L$ is a list-assignment for a graph $G$, such that $\deg_G(v) +2 \leq 3 |L(v)|$ for each vertex $v$ of $G$, 
and $\phi$ is an $L$-colouring  of $G$ that minimises the number of monochromatic edges, then $\phi$ has defect 2. Moreover, for each vertex $v$ with defect 2 under $\phi$, there is a colour $\beta_v\in L(v)\setminus\{\phi(v)\}$, such that at most two neighbours of $v$ are coloured $\beta_v$ under $\phi$.
\end{lem}

\begin{proof}
Suppose that some vertex $v$ coloured $\alpha$ is adjacent to at least three vertices also coloured $\alpha$. 
Since $\deg(v)< 3|L(v)|$, some colour $\beta\in L(v)\setminus\{\alpha\}$ is assigned to at most two neighbours of $v$. Recolouring $v$ by $\beta$ reduces the number of monochromatic edges. This contradiction shows that every vertex has defect at most $2$. 

Consider a vertex $v$ coloured $\alpha$ with defect 2. Suppose that $v$ has at least three neighbours coloured $\beta$ for each $\beta\in L(v)\setminus\{\alpha\}$. Thus $\deg(v)\geq 2 + 3 ( |L(v)|-1)$, implying 
$\deg(v) + 1 \geq 3 |L(v)|$, which is a contradiction. Thus some colour $\beta\in L(v)\setminus\{\alpha\}$ is assigned to at most two neighbours of $v$. 
\end{proof}

Given a colouring $\phi$ of a graph $G$, let $G[\phi]$ denote the monochromatic subgraph of $G$ given $\phi$. The idea for the following lemma is by  \citet[Lemma~2.6]{HST03}, adapted here for the setting of list-colourings.

\begin{lem}
\label{fixed2listcolouring}
If $H$ is a bipartite graph with bipartition $(X,Y)$ and $L$ is a list-assignment for $H$ such that $|L(v)|=2$ for all $v\in X$ and $|L(v)|=1$ for all $v\in Y$ and every $L$-colouring $\phi$ has defect 2, then $H$ has an $L$-colouring $\phi$ such that every connected subgraph of $H[\phi]$ at most two vertices in $X$.
\end{lem}

\begin{proof}
We begin by orienting the edges of $H$ so that for every vertex $v\in V(H)$ and every colour $c\in L(v)$, $v$ has at most one out-neighbour $w$ with $c\in L(w)$ and $v$ has at most one in-neighbour $w$ with $c\in L(w)$. Let $L(H)$ be the union of the lists of all vertices of $H$. For each colour $c\in L(H)$, let $H_c$ be the subgraph of $H$ induced by the vertices $w\in V(H)$ with $c\in L(w)$. There is an $L$-colouring which assigns each vertex of $H_c$ the colour $c$, so $\Delta(H_c)\leq 2$. Also, since every edge of $H$ has an endpoint $y\in Y$ and $|L(y)|=1$, evey edge of $H$ is in $E(H_c)$ for at most one $c\in L(H)$. For each $c\in L(H)$, orient the edges of $H_c$ so that no vertex has more than one in-neighbour or out-neighbour (possible since $\Delta(H_c)\leq 2$). Orient all remaining edges of $H$ arbitrarily.

We now construct an $L$-colouring $\phi$. First, colour each vertex in $Y$ with the unique colour in its list. Now run the following procedure, initialising $i:=1$.
\begin{enumerate}[{\bf 1:}]
\item If $i>|X|$, then exit.

\item Select $v_i\in X\setminus \{v_i:i\in [i-1]\}$ and select $\phi(v_i)\in L(v_i)$ arbitrarily. Increment $i$ by $1$ and go to {\bf 3}. 

\item  If there is a directed path $v_{i-1}yx$ such that $x\in X\setminus \{v_i:i\in [i-1]\}$ and $\phi(v_{i-1})=\phi(y)$ and $\phi(v_{i-1})\in L(x)$, let $v_i:=x$, select $\phi(v_i)\in L(v_i)\setminus \{\phi(v_{i-1})\}$, increment $i$ by $1$ and go to {\bf 3}. Otherwise go to {\bf 4}.

\item If there is a directed path $xyv_{i-1}$ such that $x\in X\setminus \{v_i:i\in [i-1]\}$ and $\phi(v_{i-1})=\phi(y)$ and $\phi(v_{i-1})\in L(x)$, let $v_i:=x$, select $\phi(v_i)\in L(v_i)\setminus \{\phi(v_{i-1})\}$, increment $i$ by $1$ and go to {\bf 3}. Otherwise go to {\bf 1}.
\end{enumerate}

Suppose for contradiction that some component $C$ of $H[\phi]$ has at least three vertices in $X$. Since $\phi$ is an $L$-colouring, $C$ has a directed subpath $x_1y_1x_2y_2x_3$ such that $\{x_1,x_2,x_3\}\subseteq X$. If $x_1$ was the first vertex in $\{x_1,x_2\}$ to be coloured, then $x_2$ was coloured next and $\phi(x_2)\neq \phi(x_1)$, a contradiction. If $x_2$ was the first vertex in $\{x_2,x_3\}$ to be coloured, then $x_3$ was coloured next and $\phi(x_3)\neq \phi(x_2)$, a contradiction. Hence, $x_2$ was coloured before $x_1$ and after $x_3$. But then $x_1$ was coloured immediately after $x_2$ and $\phi(x_1)\neq \phi(x_2)$, a contradiction.
\end{proof}

We now prove our first result for clustered choosability of graphs with given maximum degree. 

\MaxDegreeChoose*

\begin{proof}
Let $k:= \ceil{\frac{\Delta+2}{3}}$.
Let $L$ be a $k$-list-assignment for $G$. 
Let $\phi$ be an $L$-colouring of $G$ that minimises the number of monochromatic edges. 
By \cref{NewDefectColour}, $\phi$ is an $L$-colouring with defect 2. 
Moreover, for each vertex $v$ with defect 2 under $\phi$, 
there is a colour $\beta_v\in L(v)\setminus\{\phi(v)\}$, 
such that at most two neighbours of $v$ are coloured $\beta_v$ under $\phi$. 
Let $L'(v):=\{\phi(v),\beta_v\}$ for each vertex $v$ with defect 2. 

Let $M$ be the monochromatic subgraph of $G$. Thus $\Delta(M)\leq 2$. 
By \cref{ApplyIndTrans}, there is a set $S\subseteq V(M)$, such that $S$ is stable in $G$, every vertex in $S$ has defect 2 under $\phi$, and the following hold:
\begin{enumerate}
\item every subpath  of $M$ with at least $3\Delta-6$ vertices that contains  a vertex with degree 1 in $M$ contains at least one vertex in $S$,
\item every subpath  of $M$ with at least $5\Delta-9$ vertices that contains  a vertex with degree 1 in $M$ contains at least two vertices in $S$, and
\item every connected subgraph $C$ of $M$ on at least $\ceil{\frac{19}{2} \Delta}-16$ vertices contains at least three vertices in $S$.
\end{enumerate}

Define a subpath of $M$ to have \emph{type 1} if it contains no vertex in $S$ and at least one vertex of degree at most 1 in $M$. Define a subpath of $M$ to have \emph{type 2} if it contains at most one vertex in $S$ and at least one vertex of degree at most 1 in $M$. Note that every path of type 1 is also of type 2, and every path of type 2 or 1 that does not contain a vertex of degree 1 in $M$ contains  a vertex of degree 0 in $M$, and hence has only one vertex. By the definition of $S$, every path of type 1 has at most $3\Delta-7$ vertices and every path of type 2 has at most $5\Delta-10$ vertices.
%component of $M-S$ is a cycle on at most $12\Delta-19$ vertices or a path on at most $\ceil{\frac{13\Delta-29}{3}}$ vertices.

Let $\mathcal{T}$ be the set of connected components of $M-S$. Let $H$ be the bipartite graph with bipartition $\{S,\mathcal{T}\}$, where $s\in S$ is adjacent to $T\in \mathcal{T}$ if and only if $s$ is adjacent to $T$ in $G$, and the colour of the vertices of $T$ is in $L'(s)$. Define $L'_H$ so that $L'_H(s):=L'(s)$ for every $s\in S$, and $L'_H(T)$ is the singleton containing the colour assigned to the vertices of $T$ for every $T\in \mathcal{T}$.

Let $\phi'_H$ be an arbitrary $L'_H$-colouring of $H$, and let $\phi'$ be the corresponding $L$-colouring of $G$. Note that every vertex of $v\in S$ is assigned a colour in $L'(v)$ and every other vertex is assigned its original colour in $\phi$. Since $S$ is a stable set and by the definition of $L'$, the number of monochromatic edges given $\phi'$ is at most the number of monochromatic edges given $\phi$. Hence by our choice of $\phi$, no $L$-colouring of $G$ yields fewer monochromatic edges than $\phi'$. Hence the monochromatic subgraph $M'$ of $G$ given $\phi'$ satisfies $\Delta(M')\leq 2$. Let $M'_H$ be the graph obtained from $M'$ by contracting each $T\in \mathcal{T}$ to a single vertex. Then $M'_H$ is isomorphic to the monochromatic subgraph of $H$ given $\phi'_H$. Since $M'_H$ is a minor of $M'$ and $\Delta(M')\leq 2$, we have $\Delta(M'_H)\leq 2$. Hence, every $L'_H$-colouring of $H$ has defect 2.

%Suppose for contradiction that some $T\in \mathcal{T}$ has three distinct neighbours $v_1,v_2$ and $v_3$ in $H$ such that $\phi'_H(T)=\phi'_H(v_1)=\phi'_H(v_2)=\phi'_H(v_3)$. Since $T$ is a connected subgraph of $M$, 

%????
%For each colour $\alpha$, consider the subgraph $H_\alpha$ induced by the vertices $v$ of $H$ with 
%$\alpha\in L'(v)$. Suppose on the contrary that some vertex $v$ has three neighbours $x,y,z$ in $H_\alpha$. 
%So $\alpha\in L'(v)\cap L'(x) \cap L'(y) \cap L'(z)$, and there are path components $P,Q,R$ of $H-S$, such that 
%$v$ and $x$ are both adjacent to $P$, 
%$v$ and $y$ are both adjacent to $Q$, and
%$v$ and $z$ are both adjacent to $R$.
%Since $v$ has at most two neighbours with colour $\alpha$ in their lists, at least two of $P,Q,R$ are equal. Without loss of generality, $P=Q$. Consider the graph $T$ obtained from $P$ by adding one edge between $P$ and each of $v,x,y$. Then $T$ is a tree in which $v,x,y$ are leaves. Thus $T$ has a vertex $r\not\in\{v,x,y\}$ of degree at least 3. In $G$, recolour each of $v,x,y$ by $\alpha$, and recolour $r$ by a colour where it has at most two neighbours of that colour. This reduces the number of monochromatic edges in $G$, which contradicts the choice of $\phi$. 
%Now assume that $\Delta( H[ \{ v\in V(H): \alpha\in L'(v) \} ] )\geq 2$ for every colour $\alpha$.

By \cref{fixed2listcolouring}, $H$ has an $L'_H$-colouring $\phi'_H$ such that no component of the monochromatic subgraph has more than two vertices in $S$. Let $\phi'$ be the corresponding $L$-colouring of $G$, and note that no component of the monochromatic subgraph $M'$ of $G$ given $\phi'$ has more than two vertices in $S$. In $\phi'$, vertices of $G-S$ keep their colour from $\phi$, and vertices $v\in S$ get a colour from $L'(v)$, so $\phi'$ is an $L$-colouring that minimises the number of monochromatic edges. 

Suppose for contradiction that some vertex in $V(G-S)$ has degree 2 in $M$ and is adjacent in $M'$ to some vertex $s\in S$ which is not its neighbour in $M$ (so $\phi'(s)\neq \phi(s)$). Then the $L'$-colouring obtained from $\phi$ by recolouring $s$ with $\phi'(s)$ is not 2-defective, a contradiction.

It follows that the largest possible monochromatic component $C$ of $M'$ is obtained either from three disjoint paths in $M$ of type 1 linked by two vertices in $S$, or is obtained from a path of type 1 and a path of type 2 linked by a vertex of $S$, or is a subgraph of $M$ that contains  at most two vertices in $S$. In each case, we have $|V(C)|\leq \ceil{\frac{19}{2}\Delta}-17$. 
\end{proof}

\section{Clustered Choosability with Absolute Bounded Clustering}
\label{AbsoluteClustering}

This section proves our results for clustered choosability of graphs with given maximum average degree  (\cref{MADclusteringB}) or given maximum degree (\cref{MaxDegreeChooseAbsolute}), where the clustering is bounded by an absolute constant. The following lemma is the heart of the proof. With $I=\emptyset$, it immediately implies \cref{MaxDegreeChooseAbsolute}.

\begin{lem}
\label{absolutelyboundedlem}
If $I$ is a stable set of vertices in a graph $G$ and $L$ is a list-assignment for $G$ such that $5|L(v)|\geq 2\deg(v)+2$ for all $v\in V(G-I)$ and $5|L(v)|\geq 2\deg(v)+1$ for all $v\in I$, then $G$ has an $L$-colouring with clustering 9. Furthermore, if $I=\emptyset$, then $G$ has an $L$-colouring with clustering $6$.
\end{lem}

\begin{proof}
Let $\mathcal{C}$ be the class of $L$-colourings $\phi$ that minimise the number of monochromatic edges. Given $\phi\in \mathcal{C}$ and $v\in V(G)$, let $L(\phi, v)$ be the set of colours $c\in L(v)$ such the colouring $\phi'$ obtained from $\phi$ by recolouring $v$ with $c$ is in $\mathcal{C}$. Note that in particular $\phi(v)\in L(\phi, v)$, and that a colour $c\in L(v)$  is in $L(\phi,v)$ if and only if $|\{w\in N(v):\phi(w)=c\}|=\deg_{G[\phi]}(v)$.

\begin{claim}
\label{all2defective}
If $\phi\in \mathcal{C}$, then $\Delta(G[\phi])\leq 2$.\end{claim}
\begin{proof}
Let $v$ be a vertex of maximum degree in $G[\phi]$. If for some colour $c\in L(v)$ we have $|\{w\in N_G(v):\phi(w)=c\}|<\deg_{G[\phi]}(v)$, then the colouring $\phi'$ obtained from $\phi$ by changing the colour of $v$ to $c$ satisfies $|E(G[\phi'])|<|E(G[\phi])|$, contradicting the assumption that $\phi\in \mathcal{C}$. Hence, $\deg_G(v)\geq \deg_{G[\phi]}(v)|L(v)|$. By assumption $|L(v)|\geq \frac{1}{5}(2\deg_G(v)+1)$, and the result follows.
\end{proof}
%Note that for every $l$-colouring $\phi$ and every vertex $v$, there is some colour $c\in l(v)$ such that $v$ has at most two neighbours in $G$ coloured $c$ by $\phi$. It follows that $\Delta(G[\phi])\leq 2$ for $\phi\in \mathcal{C}$.

\begin{claim}
\label{intersectingsublists} 
If $\{\phi,\phi'\}\subseteq \mathcal{C}$, $v\in V(G-I)$ and $\deg_{G[\phi]}(v)=\deg_{G[\phi']}(v)=2$, then $|L(\phi, v)\cap L(\phi',v)|\geq 2$.\end{claim}

\begin{proof}
Suppose for contradiction that $|L(\phi, v)\cap L(\phi',v)|\leq 1$. Note that $L(\phi, v)\cup L(\phi',v)\subseteq L(v)$. Given that $|L(\phi, v)|+|L(\phi',v)|=|L(\phi, v)\cup L(\phi',v)|+|L(\phi, v)\cap L(\phi',v)|\leq |L(v)|+1$, we have $|L(\phi, v)|\leq (|L(v)|+1)/2$ without loss of generality. Since $\phi\in\mathcal{C}$, for every colour $c\in L(v)$, the vertex $v$ has at least two neighbours in $G$ coloured $c$ by $\phi$ (or else recolouring $v$ with $c$ would yield a colouring $\phi'$ with $|E(G[\phi'])|<|E(G[\phi])|$). For every colour $c\in L(v)\setminus L(\phi,v)$, the vertex $v$ has at least three neighbours coloured $c$ by $\phi$. Hence, $\deg(v)\geq 3|L(v)|-(|L(v)|+1)/2$, meaning $|L(v)|\leq \frac{1}{5}(2\deg(v)+1)$, a contradiction.
\end{proof}

Choose $\phi_0\in \mathcal{C}$ and $S\subseteq V(G-I)$ such that $S$ is a stable set in $G[\phi_0]$, every vertex in $S$ has degree 2 in $G[\phi]$, and subject to this $|S|$ is maximised. Let $S:=\{s_1,s_2,\dots,s_t\}$. For $i\in [t]$, define $\phi_i$ recursively so that $\phi_i(v)=\phi_{i-1}(v)$ for $v\in V(G)\setminus \{s_i\}$ and $\phi_i(s_i)\in (L(\phi_0,s_i)\cap L(\phi_{i-1},s_i))\setminus \{\phi_0(s_i)\}$. Such $L$-colourings exist by \cref{intersectingsublists}.

Define $L'(v):=\{\phi_0(v), \phi_t(v)\}$ for all $v\in V(G)$. 

\begin{claim}
\label{2neighbours}
If $\phi$ is an $L'$-colouring of $G$ and $s\in S$, then $|N_{G[\phi]}(s)\setminus S|=2$.
\end{claim}

\begin{proof}
Note that $L'(v)=\{\phi_0(v)\}$ for $v\in V(G)\setminus S$. Hence $|N_{G[\phi]}(s)\setminus S|=|N_{G[\phi_0]}(s)\setminus S|=2$ if $\phi(s)=\phi_0(s)$. Now suppose that $\phi(s)=\phi_t(s)$. By construction, $\phi_t(s)\in L(\phi_0,s)$, so the colouring $\phi'$ obtained from $\phi_0$ by changing the colour of $s$ to $\phi_t(s)$ is in $\mathcal{C}$. Now $\Delta(G[\phi'])\leq 2$ by \cref{all2defective}, so no vertex $s'\in S$ is adjacent to $s$ in $G[\phi']$, since $s'$ already has two neighbours in $G[\phi_0]-S$ and hence in $G[\phi']-S$. Since $|E(G[\phi'])|=|E(G[\phi_0])|$, we have $\deg_{G[\phi']}(s)=\deg_{G[\phi_0]}(s)=2$. Hence $|N_{G[\phi]}(s)\setminus S|=|N_{G[\phi']}(s)\setminus S|=\deg_{G[\phi']}(s)=2$.
\end{proof}

\begin{claim}
\label{nobadcolourings}
If $\phi$ is an $L'$-colouring of $G$, then $\phi\in\mathcal{C}$.
\end{claim}

\begin{proof}
Suppose for contradiction that for some $\{v,w\}\subseteq S$, $vw\in E(G[\phi])$. Since $S$ is a stable set in $G[\phi_0]$, either $\phi(v)=\phi_t(v)$ or $\phi(w)=\phi_t(w)$. 

If $\phi(v)=\phi_t(v)$ and $\phi(w)=\phi_t(w)$, then $v$ has three neighbours in $G[\phi_t]$ by \cref{2neighbours}. But since $\phi_i(s_i)\in L(\phi_{i-1},s_i)$ for $i\in [t]$, we have $\phi_t\in \mathcal{C}$, a contradiction. 

Hence, without loss of generality, $\phi(v)=\phi_0(v)$ and $\phi(w)=\phi_t(w)$. Now $\phi_t(w)\in L(\phi_0,w)$, so the colouring $\phi'$ obtained from $\phi_0$ by recolouring $w$ with $\phi_t(w)$ is in $\mathcal{C}$. Note $vw\in  E(G[\phi'])$ by assumption. By \cref{2neighbours}, $|N_{G[\phi']}(v)\setminus S|=|N_{G[\phi_0]}(v)\setminus S|=2$, so $\deg_{G[\phi']}(v)=3$, contradicting \cref{all2defective}.

Now $|E(G[\phi])|=|E(G[\phi]- S)]|+2|S|$ by \cref{2neighbours}. But $G[\phi]- S=G[\phi_0]- S$, so $|E(G[\phi])|=|E(G[\phi_0])|$, and $\phi\in \mathcal{C}$.
\end{proof}

Let $\mathcal{T}$ be the set of components of $G[\phi_0]-S$. Let $H$ be the bipartite graph with bipartition $(S,\mathcal{T})$ such that $s\in S$ is adjacent to $T\in \mathcal{T}$ if $s$ is adjacent to $T$ in $G$ and the colour assigned to the vertices of $T$ by $\phi_0$ is in $L'(s)$. Let $L'_H$ be the natural restriction of $L'$ to $H$. Note that an $L'_H$-colouring $\phi_H$ of $H$ corresponds to an $L'$-colouring of $G$, and $H[\phi_H]$ is a minor of $G[\phi]$, which means $\Delta(H[\phi_H])\leq 2$ by \cref{all2defective,nobadcolourings}. Hence, by \cref{fixed2listcolouring}, $H$ has an $L'_H$-colouring $\phi_H$ such that no component of $H[\phi_H]$ has more than two vertices in $S$. Let $\phi$ be the corresponding $L'$-colouring of $G$. Note that each component of $G[\phi]$ has at most two vertices in $S$. 

%Let $S'$ be a maximal independent set of vertices of degree 2 in $G[\phi]$. %Note that $|S'|=|S|$ since $S$ satisfies these requirements and $\phi_0$ and $S$ were chosen to maximise $|S|$. Hence, every component of $G[\phi]$ has at most three vertices in $S'$. 

Suppose for contradiction that some component $C$ of $G[\phi]$ has at least ten vertices. Now $\Delta(G[\phi])\leq 2$ by \cref{all2defective,nobadcolourings}, so $C$ is a cycle or a path. Hence $C$ has an induced subpath $P:=p_1p_2\dots p_{8}$ such that every vertex of $P$ has degree 2 in $G[\phi]$. Since $I$ is a stable set in $G$, at most one vertex in each of $\{p_1,p_2\}$, $\{p_4,p_5\}$ and $\{p_7,p_8\}$ is in $I$, so $C-I$ contains a stable set $S_C$ of size 3 such that every vertex of $S_C$ has degree 2 in $G[\phi]$. Define $S':=(S\setminus V(C))\cup S_C$. Since $|S\cap V(C)|\leq 2$, we have $|S'|>|S|$. However $S'\subseteq V(G-I)$, $S'$ is a stable set in $G[\phi]$, and every vertex of $S'$ has degree 2 in $G[\phi]$, contradicting our choice of $\phi_0$ and $S$.

Finally, suppose for contradiction that $I=\emptyset$ and some component $C$ of $G[\phi]$ has at least seven vertices. As before, $C$ is either a cycle or a path, so there is a stable set $S_C$ in $C$ of size 3 such that every vertex in $S_C$ has degree 2 in $G[\phi]$. Define $S':=(S\setminus V(C))\cup S_C$. Since $|S\cap V(C)|\leq 2$, we have $|S'|>|S|$. However $S'\subseteq V(G-I)$, $S'$ is a stable set in $G[\phi]$ and every vertex of $S'$ has degree 2 in $G[\phi]$, contradicting our choice of $\phi_0$ and $S$.
\end{proof}

The following lemma is analogous to \cref{Extend}. 

\begin{lem}
\label{extensioncor}
Let $(A,B)$  be a partition of the vertex set of a graph $G$, 
let $I\subseteq B$ be a stable set, and let $L$ a list-assignment for $G$. 
If $5|L(v)|-5\deg_A(v)\geq 2\deg_B(v)+2$ for all $v\in B\setminus I$
 and $5|L(v)|-5\deg_A(v)\geq 2\deg_B(v)+1$ for all $v\in I$, 
 then every $L$-colouring of $G[A]$ with clustering 9 can be extended 
 to an $L$-colouring of $G$ with clustering 9.
\end{lem}

\begin{proof}
Let $\phi$ be an $L$-colouring of $G[A]$  with clustering 9. 
For each vertex $v\in B$, let $L'(v):= L(v)\setminus \{\phi(x):x\in N_A(v\}$. 
Thus 
$|L'(v)| \geq |L(v)|  - \deg_A(v) \geq \frac25 (\deg_B(v)+1)$ for $ v \in B\setminus I$, and
$|L'(v)| \geq |L(v)|  - \deg_A(v) \geq \frac15 (2\deg_B(v)+1)$ for $ v \in I$. 
\cref{absolutelyboundedlem} implies that $G[B]$ is $L$-colourable with  clustering 9.
By construction, there is no monochromatic edge between $A$ and $B$. 
Thus  $G$ is $L$-colourable with clustering 9.
\end{proof}

We now prove the main result of this section. 

\MADclusteringB*

\begin{proof}
Let $k:=\floor{\frac{7}{10}\mad(G)}+1$. We proceed by induction on $|V(G)|$. The claim is trivial if $|V(G)|\leq 9$. Assume that $|V(G)|\geq 10$. Let $L$ be a $k$-list-assignment for $G$. 

Let $p$ be the maximum integer for which there are pairwise disjoint sets $X_1,\dots,X_p \subseteq V(G)$, such that 
for each $i\in[p]$, we have $|X_i|\in\{1,2\}$, and if $A_i := X_1\cup\dots\cup X_{i-1}$ and $B_i:=V(G)\setminus A_i$, then at least one of the following conditions holds:
\begin{itemize}
\item $X_i=\{v_i\}$ and $5|L(v_i)| \leq 5\deg_{A_i}(v_i) + 2\deg_{B_i}(v_i)$, or
\item $X_i=\{v_i,w_i\}$ and $v_iw_i\in E(G)$ and $5|L(v_i)| \leq 5\deg_{A_i}(v_i) + 2\deg_{B_i}(v_i)+1$ and 
$5|L(w_i)| \leq 5\deg_{A_i}(w_i) + 2\deg_{B_i}(w_i)+1$.
\end{itemize}

%Let $A_i:=\bigcup_{j=1}^{i-1} X_j$ and $B_i:=V(G)\setminus A_i$ (So $A_1=\emptyset$ and $B_1=V(G)$). 
%Halt if $A_i=V(G)$. 
%Since $\mad(G[A_i]) \leq \mad(G)$, by induction, $G[A_i]$ is $L$-colourable with clustering $9$. \\
%Case (A): If $5|L(v)| \leq 5\deg_{A_i}(v) + 2\deg_{B_i}(v)$ for some vertex $v\in B_i$, then let $X_i:=\{v_i\}:=\{v\}$. \\
%%Case (B): If $\deg_G(v)\geq 3k$ for some vertex $v\in B$,  then let $X_i:=\{v_i\}:=\{v\}$.\\
%Case (B): If Case(A) does not hold and $5|L(v)| \leq 5\deg_{A_i}(v) + 2\deg_{B_i}(v)+1$ and $5|L(w)| \leq 5\deg_{A_i}(w) + 2\deg_{B_i}(w)+1$ 
%for some edge $vw\in E(G[B_i])$, then let $X_i:=\{v_i,w_i\}:=\{v,w\}$. 

%If (A) or (B) is satisfied, then increment $i$ and continue the procedure. 
%Otherwise, let $I$ be the set of vertices $v\in B_i$ for which 
%$5|L(v)| \leq 5\deg_{A_i}(v) + 2\deg_{B_i}(v)+1$.  
%Then $I$ is a stable set satisfying \cref{StableSetExtension}, implying $G$ is $L$-colourable with clustering 9. 

First suppose that $X_1\cup\dots\cup X_p\neq V(G)$. 
Let $A := X_1\cup\dots\cup X_p$ and $B:=V(G)\setminus A$. 
We now show that \cref{extensioncor} is applicable. 
By the maximality of $p$, each vertex $v\in B$ satisfies 
$5|L(v)| \geq 5\deg_{A}(v) + 2\deg_{B}(v) +1$.
Let $I$ be the set of vertices $v\in B$ for which $5|L(v)| = 5\deg_{A}(v) + 2\deg_{B}(v)+1$.  
By the maximality of $p$, $I$ is a stable set.
Since $\mad(G[A]) \leq \mad(G)$, by induction, $G[A]$ is $L$-colourable with clustering 9. 
By \cref{extensioncor}, $G$ is $L$-colourable with clustering 9. 

Now assume that $X_1\cup\dots\cup X_p= V(G)$. 
Let $R:=\{i\in[p]:|X_i|=1\}$ and $S:=\{i\in[p]:|X_i|=2\}$.  
Thus
\begin{align*}
5k |V(G)|  & 
 \leq 
 \sum_{i\in R} (3\deg_{A_i}(v_i) + 2\deg_G(v_i)) + \\
& \quad\; \sum_{i\in S} (3\deg_{A_i}(v_i) + 2\deg_G(v_i) + 1 + 3\deg_{A_i}(w_i) + 2\deg_G(w_i) + 1) \\
& \leq
3\sum_{i\in R} \deg_{A_i}(v_i) + 3\sum_{i\in S} (\deg_{A_i}(v_i) + \deg_{A_i}(w_i) + 1)  + 2\sum_{v\in V(G)} \deg_G(v) \\
& =  7|E(G)| .
\end{align*}
Hence
$\frac{10}{7} k \leq  \frac{2|E(G)|}{|V(G)|} \leq \mad(G)$, implying
$k \leq  \frac{7}{10} \mad(G)$, which is a contradiction. 
\end{proof}

%%%%%%%%%%%%%%%%%%%%%%
\section{Clustered Choosability and Maximum Average Degree }
\label{ClusteredChoosabilityMAD}

This section proves our final results for clustered choosability of graphs with given maximum average degree (\cref{MADclusteringD,MADclusteringExtension}). 

\begin{lem}
\label{StableSet}
If $I$ is a stable set in a graph $G$ of maximum degree $\Delta\geq 3$, and $L$ is a list-assignment of $G$, and $3|L(v)| \geq\deg_G(v)+1$ for each vertex $v\in I$, and $3|L(v)| \geq\deg_G(v)+2$ for each vertex $v\in V(G)\setminus I$, then $G$ is $L$-colourable with clustering $19\Delta-32$. 
\end{lem}

\begin{proof}
Let $\phi$ be an $L$-colouring of $G$ that minimises the number of monochromatic edges. 
By \cref{NewDefectColour}, $\phi$ is an $L$-colouring with defect 2. 
Moreover, for each vertex $v\in V(G)\setminus I$ with defect 2 under $\phi$, 
there is a colour $\beta_v\in L(v)\setminus\{\phi(v)\}$, 
such that at most two neighbours of $v$ are coloured $\beta_v$ under $\phi$. 
Let $L'(v):=\{\phi(v),\beta_v\}$ for each vertex $v\in V(G)\setminus I$ with defect 2. 

Let $M$ be the monochromatic subgraph of $G$. Thus $\Delta(M)\leq 2$. Each component of $M$ is a cycle or path. Orient each cycle component of $M$ to become a directed cycle, and orient each path component of $M$ to become a directed path. 

Let $G'$ be obtained from $G$ as follows: first delete all non-monochromatic edges incident to all vertices in $I$. Note that vertices in $I$ now have degree at most 2. Now if $vx$ is a directed monochromatic edge in $G$ with $x\in I$ and $x$ having defect 2, then contract $vx$ into a new vertex $v'$. Note that $v\in V(G)\setminus I$ since $I$ is a stable set. Note also that $\Delta(G')\leq \Delta(G)\leq \Delta$. Consider $v'$ to be coloured by the same colour as $v$. Let $M_{G'}$ be the monochromatic subgraph of $G'$. Then $M_{G'}$ is obtained from $M$ by the same set of contractions that formed $G'$ from $G$, and $M_{G'}$ is an induced subgraph of $G'$ with maximum degree at most 2. 

By \cref{ApplyIndTrans}, there is a set $S'\subseteq V(M_{G'})$, such that $S'$ is stable in $G$, every vertex in $S'$ has defect 2 under $\phi$, and the following hold:
\begin{enumerate}
\item every subpath  of $M_{G'}$ with at least $3\Delta-6$ vertices that contains  a vertex with degree 1 in $M$ contains at least one vertex in $S'$,
\item every subpath  of $M_{G'}$ with at least $5\Delta-9$ vertices  that contains  a vertex with degree 1 in $M$ contains at least two vertices in $S'$, and 
\item every connected subgraph $C$ of $M_{G'}$ with at least $\ceil{\frac{19}{2} \Delta}-16$ vertices contains at least three vertices in $S'$.
\end{enumerate}

Let $S$ be obtained from $S'$ by replacing each vertex $v'$ (arising from a contraction) by the corresponding vertex $v$ in $G$. Thus $S\cap I=\emptyset$. By construction,  $S$ is a stable set in $G$, every vertex in $S$ has defect 2 under $\phi$, and each of the following hold:

\begin{enumerate}
\item every subpath  of $M$ with at least $6\Delta-12$ vertices contains  a vertex with degree 1 in $M$ contains at least one vertex in $S$,
\item every subpath  of $M$ with at least $10\Delta-18$ vertices contains  a vertex with degree 1 in $M$ contains at least two vertices in $S$,
\item every connected subgraph $C$ of $M$ with at least $19\Delta-31$ vertices contains at least three vertices in $S$.
\end{enumerate}
Define a subpath  of $M$ to have \emph{type 1} if it contains no vertex in $S$ and at least one vertex of degree at most 1 in $M$. Define a subpath  of $M$ to have \emph{type 2} if it contains at most one vertex in $S$ and at least one vertex of degree at most 1 in $M$. Note that every path of type 1 is also of type 2, and that any path of type 2 or 1 that contains no vertex of degree 1 in $M$ contains a vertex of degree 0 in $M$, and hence has only one vertex. By the definition of $S$, every path of type 1 has at most $6\Delta-13$ vertices and every path of type 2 has at most $10\Delta-19$ vertices.

%By \cref{ApplyIndTrans} with $a_0=6$ applied to $G'$, there is a set $S'\subseteq V(M')$, such that $S'$ is stable in $G'$, every vertex in $S'$ has degree 2 in $M'$, and each component of $M'-S'$ is a path of at most $\ceil{\frac{26\Delta(G)-58}{6}}$ vertices or a cycle of at most $12\Delta-19$ vertices.  

Let $\mathcal{T}$ be the set of connected components of $M-S$, and define a bipartite graph $H$ with bipartition $\{S,\mathcal{T}\}$, where $s\in S$ is adjacent to $T\in \mathcal{T}$ if and only if $s$ is adjacent to $T$ in $G$, and the colour of the vertices of $T$ is in $L'(s)$. Define $L'_H$ so that $L'_H(s):=L'(s)$ for every $s\in S$, and $L'_H(T)$ is the singleton containing the colour assigned to the vertices of $T$ for every $T\in \mathcal{T}$.

Let $\phi'_H$ be an arbitrary $L'_H$-colouring of $H$, and let $\phi'$ be the corresponding $L$-colouring of $G$. Note that every vertex of $v\in S$ is assigned a colour in $L'(v)$ and every other vertex is assigned its original colour in $\phi$. Since $S$ is a stable set and by the definition of $L'$, the number of monochromatic edges given $\phi'$ is at most the number of monochromatic edges given $\phi$. Hence by our choice of $\phi$, no $L$-colouring of $G$ yields fewer monochromatic edges than $\phi'$. Hence the monochromatic subgraph $M'$ of $G$ given $\phi'$ satisfies $\Delta(M')\leq 2$. Let $M'_H$ be the graph obtained from $M'$ by contracting each $T\in \mathcal{T}$ to a single vertex. Then $M'_H$ is isomorphic to the monochromatic subgraph of $H$ given $\phi'_H$. Since $M'_H$ is a minor of $M'$, we have $\Delta(M'_H)\leq 2$. Hence, every $L'_H$-colouring of $H$ has defect 2.

By \cref{fixed2listcolouring}, $H$ has an $L'_H$-colouring $\phi'_H$ such that no component of the monochromatic subgraph has more than two vertices in $S$. Let $\phi'$ be the corresponding $L$-colouring of $G$, and note that no component of the monochromatic subgraph $M'$ of $G$ given $\phi'$ has more than two vertices in $S$. In $\phi'$, vertices of $G-S$ keep their colour from $\phi$, and vertices $v\in S$ get a colour from $L'(v)$, so $\phi'$ is an $L$-colouring which minimises the number of monochromatic edges.

Suppose for contradiction that some vertex in $V(G-S)$ has degree 2 in $M$ and is adjacent in $M'$ to some vertex $s\in S$ which is not its neighbour in $M$ (so $\phi'(s)\neq \phi(s)$). Then the $L'$-colouring obtained from $\phi$ by recolouring $s$ with $\phi'(s)$ is not 2 defective, a contradiction.

It follows that the largest possible monochromatic component $C$ of $M'$ is obtained either from three disjoint paths in $M$ of type 1 linked by two vertices in $S$, or is obtained from a path of type 1 and a path of type 2 linked by a vertex of $S$, or is a subgraph of $M$ that contains  at most two vertices in $S$. In each case, we have $|V(C)|\leq 19\Delta-32$. 
%
% The largest possible monochromatic component in $G$ is either obtained from three disjoint path components in $M-S$ linked together by two vertices in $S$, and hence has at most $3(2\ceil{\frac{22\Delta-49}{5}}+1)+2< 26\Delta-48$ vertices, or is a cycle of at most $24\Delta-38$ vertices. Since $\Delta\geq 5$, we have $24\Delta-38\leq 26\Delta-48$. 
\end{proof}
%???

We have the following analogue of \cref{Extend,extensioncor}.

%\comment{What if we relax the condition in \cref{StableSet} to $I$ is a matching or a bounded degree subgraph? Perhaps we could prove 2-choosability with bounded clustering for max degree 5 graphs?}

\begin{lem}
\label{StableSetExtension}
For a graph $G$, let $A,B$ be a partition of $V(G)$ with $\Delta:=\Delta(G[B])\geq 3$, and let $I$ be a stable set of $G$ contained in $B$. 
Let $L$ be a list-assignment for $G$ and let $c$ be an integer such that $c\geq 19\Delta-32$, 
$G[A]$ is $L$-colourable with  clustering $c$, 
$3|L(v)| \geq 3\deg_A(v) + \deg_B(v)+1$ for each vertex $v\in I$, and
$3|L(v)| \geq 3\deg_A(v) + \deg_B(v)+2$ for each vertex $v\in B\setminus I$.
Then $G$ is $L$-colourable with clustering $c$.
\end{lem}

\begin{proof}
Let $\phi$ be an $L$-colouring of $G[A]$ with  clustering $c$. 
For each vertex $v\in B$, let $L'(v):= L(v)\setminus \{\phi(x):x\in N_A(v\}$. 
Thus $|L'(v)| \geq |L(v)| - \deg_A(v)$, implying 
$3|L'(v)| \geq \deg_B(v)+1$ for each vertex $v\in I$, and
$3|L'(v)| \geq \deg_B(v)+2$ for each vertex $v\in B\setminus I$.
\cref{StableSet} implies that $G[B]$ is $L$-colourable with  clustering $19\Delta-32$. 
By construction, there is no monochromatic edge between $A$ and $B$. 
Thus  $G$ is $L$-colourable with clustering $c$.
\end{proof}
 
We now prove \cref{MADclusteringExtension}, which implies \cref{MADclusteringD} when $n_0=1$.

\MADclusteringExtension*

\begin{proof}
We first prove the $k=1$ case. Let $G$ be a graph with $\mad(G,n_0)<\frac{3}{2}$. Every component of a graph with maximum average degree less than $\frac32$ has at most three vertices. Thus every component of $G$ has at most $\max\{n_0-1,3\}$ vertices. Hence, every $1$-list-assignment has clustering $\max\{n_0-1,3\}\leq c$. Now assume that $k\geq 2$. 

We proceed by induction on $|V(G)|$. Let $L$ be a $k$-list-assignment for $G$. If $|V(G)|\leq n_0-1$, then colour each vertex $v$ by a colour in $L(v)$, so that each colour is used at most $\lceil \frac{n_0-1}{k}\rceil$ times. We obtain an $L$-colouring with clustering  $\ceil{ \frac{n_0-1}{k}}$. 
Now assume that $|V(G)|\geq n_0$. 

Let $p$ be the maximum integer for which there are pairwise disjoint sets $X_1,\dots,X_p \subseteq V(G)$, such that 
for each $i\in[p]$, we have $|X_i|\in\{1,2\}$, and if $A_i := X_1\cup\dots\cup X_{i-1}$ and $B_i:=V(G)\setminus A_i$, then at least one of the following conditions hold:
\begin{itemize}
\item $X_i=\{v_i\}$ and $3|L(v_i)| \leq 3\deg_{A_i}(v_i) + \deg_{B_i}(v_i)$, or
\item $X_i=\{v_i,w_i\}$ and $v_iw_i\in E(G)$ and $3|L(v)| \leq 3\deg_{A_i}(v) + \deg_{B_i}(v)+1$ and $3|L(w)| \leq 3\deg_{A_i}(w) + \deg_{B_i}(w)+1$.
\end{itemize}

%
%The following procedure either determines the desired $L$-colouring of $G$, or constructs a partition $X_1,\dots,X_p$ of $V(G)$ with  $|X_i|\in\{1,2\}$ for each $i\in[p]$. 
%
%\begin{enumerate}[{\bf 1:}]
%\item Initialise $i:=1$. 
%
%\item Let $A_i:=\bigcup_{j=1}^{i-1} X_j$ and $B_i:=V(G)\setminus A_i$. (So $A_1=\emptyset$ and $B_1=V(G)$.) If $A_i=V(G)$ then exit.
%
%\item Case (A): If $3|L(v)| \leq 3\deg_{A_i}(v) + \deg_{B_i}(v)$ for some vertex $v\in B_i$, then let $X_i:=\{v_i\}:=\{v\}$, increment $i$,  and go to \textbf{2}; otherwise go to \textbf{3}. 
%
%\item Case (B): If Case(A) does not hold and $3|L(v)| \leq 3\deg_{A_i}(v) + \deg_{B_i}(v)+1$ and $3|L(w)| \leq 3\deg_{A_i}(w) + \deg_{B_i}(w)+1$ for some edge $vw\in E(G[B])$, then let $X_i:=\{v_i,w_i\}:=\{v,w\}$, increment $i$,  and go to \textbf{2}; otherwise go to \textbf{4}. 

First suppose that $X_1\cup\dots\cup X_p \neq V(G)$. 
Let $A := X_1\cup\dots\cup X_p$ and $B:=V(G)\setminus A$. 
Since $\mad(G[A],n_0) \leq \mad(G,n_0)$, by induction, $G[A]$ is $L$-colourable with clustering $c$. 
We now show that \cref{StableSetExtension} is applicable. 
By the maximality of $p$, for each $v\in B$, 
$$3k=3|L(v)| \geq 3\deg_{A}(v) + \deg_{B}(v) + 1 \geq  \deg_{B}(v) + 1 .$$ 
Let $\Delta:=3k-1$. Then $\Delta(G[B])\leq 3k-1=\Delta$. 
Since $k\geq 2$, we have $\Delta\geq 5$ and $19\Delta-32=19(3k-1)-32=57k-51\leq c$. 
Let $I$ be the set of vertices $v\in B$ for which $3|L(v)| = 3\deg_{A}(v) + \deg_{B}(v)+1$. 
By the maximality of $p$, $I$ is a stable set. 
\cref{StableSetExtension} thus implies that $G$ is $L$-colourable with clustering $c$. 

Now assume that $X_1\cup\dots\cup X_p =V(G)$. 
Let $R:=\{i\in[p]:|X_i|=1\}$ and $S:=\{i\in[p]:|X_i|=2\}$.  
For $i\in R$, condition (A) holds, implying $3k \leq 2\deg_{A_i}(v_i) + \deg_G(v_i)$.
For $i\in S$, condition (B) holds, implying $3k \leq 2\deg_{A_i}(v_i) + \deg_G(v_i)+1$ and $3k \leq 2\deg_{A_i}(w_i) + \deg_G(w_i)+1$. Thus
\begin{align*}
3k |V(G)|  & 
 \leq 
 \sum_{i\in R} (2\deg_{A_i}(v_i) + \deg_G(v_i)) + \\
& \quad\; \sum_{i\in S} (2\deg_{A_i}(v_i) + \deg_G(v_i) + 1 + 2\deg_{A_i}(w_i) + \deg_G(w_i) + 1) \\
& = 
2 \sum_{i\in R} \deg_{A_i}(v_i) + 2 \sum_{i\in S} (\deg_{A_i}(v_i) + \deg_{A_i}(w_i) + 1)  + \sum_{v\in V(G)} \deg_G(v) \\
& =  4|E(G)| .
\end{align*}
Hence
$\frac{3}{2} k \leq  \frac{2|E(G)|}{|V(G)|} \leq \mad(G)$, and $|V(G)|\geq n_0$ implying
$k \leq  \frac23 \mad(G,n_0)$, which is a contradiction. 
\end{proof}

%%%%%%%%%%%%%%%%%%%%%%%%%
\section{Earth-Moon Colouring and Thickness}
\label{EarthMoon}

The union of two planar graphs is called an \emph{earth-moon} (or \emph{biplanar}) graph. The famous earth-moon problem asks for the maximum chromatic number of earth-moon graphs \citep{ABG11,Ringel59,BGS08,GS09,JR00,Hut93}. It follows from Euler's formula that every earth-moon graph has maximum average degree less than 12, and is thus 12-colourable. On the other hand, there are 9-chromatic earth-moon graphs \citep{BGS08,GS09}. So the maximum chromatic number of earth-moon graphs is 9, 10, 11 or 12. 

Defective and clustered colourings provide a way to attack the earth-moon problem. First consider defective colourings of earth-moon graphs. Since the maximum average degree of every earth-moon graph is less than $12$, \cref{HavetSereni} by \citet{HS06} implies that every earth-moon graph is $k$-choosable with defect $d$, for $(k,d)\in\{(7,18),(8,9),(9,5),(10,3),(11,2)\}$. This result gives no bound with at most $6$ colours. \citet{OOW} went further and showed that every earth-moon graph is $k$-choosable with defect $d$, for $(k,d)\in\{(5,36),(6,19),(7,12),(8,9),(9,6),(10,4),(11,2)\}$. Examples show that 5 colours is best possible \citep{OOW}. Thus the defective chromatic number of earth-moon graphs equals 5. \cref{MADdefect} implies that every earth-moon graph is $k$-choosable with defect $d$ for $(k,d)\in\{(7,6),(8,3),(9,2),(11,1)\}$. These results improve the best known bounds when $k\in\{7,8,9,11\}$. 

Now consider clustered colouring of earth-moon graphs. \citet{WoodSurvey} describes examples of earth-moon graphs that are not 5-colourable with bounded clustering. Thus the clustered chromatic number of earth-moon graphs is at least 6. \cref{KY} by \citet{KY17} proves that earth-moon graphs are 9-colourable with clustering 2. Other results for clustered colouring do not work for earth-moon graphs since they can contain expanders \citep{DSW16}, and thus do not have sub-linear separators. Since every earth-moon graph has maximum average degree strictly less than $12$, \cref{MADclusteringA,MADclusteringD} imply the following:

%\citet{WoodSurvey} states that closing this gap is an interesting problem because the existing methods say nothing for graphs with given thickness. For example, the 1-subdivision of $K_n$ has thickness 2. Thus thickness 2 graphs have unbounded $\nabla$. Similarly, \cref{SeparatorIsland} is not applicable since graphs with thickness 2 do not have sublinear balanced separators. Indeed, \citet{DSW16} constructed `expander' graphs with thickness 2, bounded degree, and with no $o(n)$ balanced separators. 

\begin{thm}
\label{EarthMoon9}
Every earth-moon graph is:
\begin{itemize}
\item $9$-choosable with clustering $2$.
\item $8$-choosable with clustering $405$.
\end{itemize}
\end{thm}

It is open whether every earth-moon graph is $6$ or $7$-colourable with bounded clustering. 

Earth-moon graphs are generalised as follows. The \emph{thickness} of a graph $G$ is the minimum integer $t$ such that $G$ is the union of $t$ planar subgraphs; see~\citep{MOS98} for a survey. It follows from Euler's formula that graphs with thickness $t$ are $(6t-1)$-degenerate and thus $6t$-colourable. For $t\geq 3$, complete graphs provide a lower bound of $6t-2$. It is an open problem to improve these bounds; see~\citep{Hut93}. \citet{OOW} studied defective colourings of graphs with given thickness, and proved the following result. 

\begin{thm}[\citep{OOW}] 
\label{DefectThickness}
The defective chromatic number of the class of graphs with thickness $t$ equals $2t+1$. In particular, every such graph is $(2t+1)$-choosable with defect $2t(4t+1)$.
\end{thm}

Now consider clustered colourings of graphs with given thickness. Obviously, the clustered chromatic number of graphs with thickness $t$ is at most $6t$, and \citet{WoodSurvey} proved a lower bound of $2t+2$. Since every graph with thickness $t$ has maximum average degree strictly less than $6t$, \cref{MADclusteringA,MADclusteringB,MADclusteringD} imply the following improved upper bounds.

\begin{thm}
\label{Thickness}
Every graph with thickness $t$ is:
\begin{itemize}
\item $\ceil{\frac{9}{2}t}$-choosable with defect 1 and clustering $2$, 
\item $\ceil{ \frac{21}{5} t}$-choosable with clustering $9$, 
\item $4t$-choosable with clustering $228t-51$.
\end{itemize}
\end{thm}

Thickness is generalised as follows; see \citep{JR00,OOW,WoodSurvey}. For an integer $g\geq 0$, the \emph{$g$-thickness} of a graph $G$ is the minimum integer $t$ such that $G$ is the union of $t$ subgraphs each with Euler genus at most $g$. 
 \citet{OOW} determined the defective chromatic number of this class as follows (thus generalising \cref{DefectThickness}).

\begin{thm}[\citep{OOW}] 
\label{DefectGenusThickness}
For integers $g\geq 0$ and $t\geq 1$, the defective chromatic number of the class of graphs with $g$-thickness $t$ equals $2t+1$. 
In particular, every such graph is $(2t+1)$-choosable with defect $2tg+8t^2+2t$. 
\end{thm}

Now consider clustered colourings of graphs with $g$-thickness $t$. \citet{WoodSurvey} proved that every such graph is $(6t+1)$-choosable with clustering $\max\{g,1\}$.  Euler's formula implies that every $n$-vertex graph with $g$-thickness $t$ has less than $3t(n+g-2)$ edges (for $n\geq 3$), implying $\mad(G,4tg-8t+1)<6t+\frac{3}{2}$. Hence, \cref{MADclusteringExtension} implies the following improvement to this upper bound. 

\begin{thm}
For $g\geq 0$ and $t\geq 1$, every graph with $g$-thickness $t$ is $(4t+1)$-choosable with clustering 
$\max \{ \ceil{ \frac{4tg-8t}{4t+1} },\,228t+6\}$.
\end{thm} 

This result highlights the utility of considering $\mad(G,n_0)$.

%%%%%%%%%%%%%%

%Vida: $2 + \sqrt{ 3g+ 3 }$

%\section{To-Do}
%
%\citet{HST03} proved that every graph with maximum degree 5 is 2-colourable with bounded clustering. Does this proof give 2-choosability?
%
%Is there a direction connection between the clustered chromatic number of graphs with max degree $\Delta$ and the  clustered chromatic number of graphs with max average degree $k$?

%%%%%%%%%%%%%%%%%%%%%%
\section{Stack and Queue Layouts}
\label{StackQueue}

This section applies our results to graphs with given stack- or queue-number. Again, previous results for clustered colouring do not work for graphs with given stack- or queue-number since they can contain expanders \citep{DSW16}, and thus do not have sub-linear separators. 

A \emph{$k$-stack layout} of a graph $G$ consists of a linear ordering
$v_1,\dots,v_n$ of $V(G)$ and a partition $E_1,\dots,E_k$ of $E(G)$
such that no two edges in $E_i$ cross with respect to $v_1,\dots,v_n$
for each $i\in[1,k]$. Here edges $v_av_b$ and $v_cv_d$  \emph{cross}
if $a<c<b<d$. A graph is a \emph{$k$-stack graph} if it has a
$k$-stack layout. The \emph{stack-number} of a graph $G$ is the
minimum integer $k$ for which $G$ is a $k$-stack graph. Stack layouts
are also called \emph{book embeddings}, and stack-number is also
called \emph{book-thickness}, \emph{fixed outer-thickness} and
\emph{page-number}. \citet{DujWoo04} showed that the maximum chromatic
number of $k$-stack graphs is in $\{2k,2k+1,2k+2\}$.

A \emph{$k$-queue layout} of a graph $G$ consists of a linear ordering
$v_1,\dots,v_n$ of $V(G)$ and a partition $E_1,\dots,E_k$ of $E(G)$
such that no two edges in $E_i$ are nested with respect to
$v_1,\dots,v_n$ for each $i\in[1,k]$. Here edges $v_av_b$ and $v_cv_d$
are \emph{nested} if $a<c<d<b$. The \emph{queue-number} of a graph $G$
is the minimum integer $k$ for which $G$ has a $k$-queue layout. A
graph is a \emph{$k$-queue graph} if it has a $k$-queue layout.
\citet{DujWoo04} showed that the maximum chromatic number of $k$-queue
graphs is in the range $[2k+1,4k]$.

Consider clustered colourings of $k$-stack and $k$-queue graphs.
\citet{WoodSurvey} noted
the clustered chromatic number of the class of $k$-stack graphs is in
$[k+2,2k+2]$, and that
the clustered chromatic number of the class of $k$-queue graphs is in
$[k+1,4k]$.
The lower bounds come from standard examples, and the upper bounds hold since
every $k$-stack graph has maximum average degree less than $2k+2$,
and every $k$-queue graph has maximum average degree less than $4k$.
\cref{MADclusteringA,MADclusteringB,MADclusteringD} thus imply the
following improved upper bounds:
\begin{thm}
Every $k$-stack graph is:
\begin{itemize}
\item $\floor{\frac{3k+4}{2}}$-choosable with defect $1$, and thus
with clustering $2$.
\item $\floor{\frac{7k+11}{5}}$-choosable with clustering $9$.
\item $\floor{\frac{4k+6}{3}}$-choosable with clustering at most $76k + 53$.
\end{itemize}
\end{thm}

\begin{thm}
Every $k$-queue graph is:
\begin{itemize}
\item $3k$-choosable with defect $1$, and thus with clustering $2$.
\item $\floor{\frac{14k+4}{5}}$-choosable with clustering $9$.
\item $\floor{\frac{8k+2}{3}}$-choosable with clustering at most $152k-13$.
\end{itemize}
\end{thm}

\subsection*{Acknowledgements} This research was initiated at the Bellairs Workshop on Graph Theory (20--27 April 2018). Many thanks to the other workshop participants for stimulating conversations and for creating a productive working environment. 

%
%\bibliographystyle{myNatbibStyle}
%\bibliography{myBibliography}

\def\soft#1{\leavevmode\setbox0=\hbox{h}\dimen7=\ht0\advance \dimen7
  by-1ex\relax\if t#1\relax\rlap{\raise.6\dimen7
  \hbox{\kern.3ex\char'47}}#1\relax\else\if T#1\relax
  \rlap{\raise.5\dimen7\hbox{\kern1.3ex\char'47}}#1\relax \else\if
  d#1\relax\rlap{\raise.5\dimen7\hbox{\kern.9ex \char'47}}#1\relax\else\if
  D#1\relax\rlap{\raise.5\dimen7 \hbox{\kern1.4ex\char'47}}#1\relax\else\if
  l#1\relax \rlap{\raise.5\dimen7\hbox{\kern.4ex\char'47}}#1\relax \else\if
  L#1\relax\rlap{\raise.5\dimen7\hbox{\kern.7ex
  \char'47}}#1\relax\else\message{accent \string\soft \space #1 not
  defined!}#1\relax\fi\fi\fi\fi\fi\fi}

\end{document}